\documentclass[a4paper,12pt]{article} 

\usepackage[T1]{fontenc}
\usepackage{amssymb}

\usepackage[dvips]{graphicx}

\usepackage{amsfonts, amsmath, amsthm, amssymb}
\usepackage[ruled,vlined,linesnumbered]{algorithm2e}
\usepackage{epsfig}
\usepackage[T1]{fontenc}
\usepackage[cp1250]{inputenc}

\usepackage{tikz}
\usepackage{fancyhdr}
\makeindex

\allowdisplaybreaks[1]

\newtheorem{theorem}{Theorem}

\newtheorem{corollary}[theorem]{Corollary}

\newtheorem{definition}[theorem]{Definition}

\newtheorem{example}[theorem]{Example}

\newtheorem{lemma}[theorem]{Lemma}

\newtheorem{proposition}[theorem]{Proposition}
\newtheorem{remark}[theorem]{Remark}

\begin{document}

\title{Integrations on rings}
\author{Iztok Bani\v c}
\maketitle

\begin{abstract}
In calculus, an indefinite integral of a function $f$ is a differentiable function $F$ whose derivative is equal to $f$. In present paper, we generalize this notion of the indefinite integral from the ring of real functions to any ring. The main goal of the paper is to focus on the properties of such generalized integrals that are inherited from the well-known basic properties of indefinite integrals of real functions.
\end{abstract}
\-


\section{Introduction}
In calculus, an antiderivative, primitive function or indefinite integral of a function $f$ is a differentiable function $F$ whose derivative is equal to $f$, i.e., $F ' = f$. In present paper we generalize this notion of the indefinite integral from the ring of real functions to any ring
(by a ring we always mean an associative ring not necessarily with an identity element unless explicitly mentioned otherwise). We will use the concept of well-studied derivations on rings (see Definition \ref{moja}) to generalize this notion of an indefinite integral from calculus. Going through the literature, one may easily notice that the derivations on rings have been the subject of intensive research for decades, while the notion of integrations on rings has never been introduced. 

\begin{definition}\label{moja}
Let $R$ be a ring. An additive mapping 
 $d:R\to R$ is a derivation on $R$, if 
$$
d(x\cdot y)=d(x)\cdot y+x\cdot d(y)
$$
for all $x,y\in R$.
\end{definition}
This concept of derivations on rings has also been generalized (as well as restricted) in many ways, for instance, there are inner derivations, Jordan derivations, Lie derivations, $(\theta ,\phi )$-derivations, generalized derivations, left derivations, and more; for examples see \cite{benkovic,benkovic1,bresarvukman1,bresarvukman2,bresar0,bresar1,fosner,hvala,hvala1,liu,vukman2}. In present paper we also deal with two such functions that are related to a derivation, see Definition \ref{moja1}. 
\begin{definition}\label{moja1}
Let $R$ be a ring. An additive mapping 
\begin{enumerate}
\item $d:R\to R$ is an inner derivation on $R$, if there is $a\in R$ such that
$$
d(x)=[x,a]=x\cdot a-a\cdot x
$$
for all $x\in R$.
\item $\delta:R\to R$ is a Jordan derivation on $R$, if 
$$
\delta(x\circ y)=\delta(x)\circ y+x\circ \delta(y)
$$
for all $x,y\in R$, where $x\circ y=x\cdot y+y\cdot x$ denotes the Jordan product on $R$.
\end{enumerate}
\end{definition}
It is a well-known fact that every inner derivation is a derivation, but there are derivations that are not inner derivations.
Also, every derivation is a Jordan derivation. It is also a well-known fact that not every Jordan derivation is a derivation. A classical result of Herstein asserts that any Jordan derivation on $2$-torsion free prime ring (recall that a ring is $n$-torsion free, where $n>1$ is an integer, if $nx=0$ implies $x=0$ for any $x\in R$) is a derivation (see also \cite{bresarvukman}). Cusak has generalized the Herstein's theorem to $2$-torsion free semi prime rings \cite{cusak} (see also \cite{bresar}). 

Derivations and Jordan derivations on rings have become very popular since their introduction and have been studied by many authors and many papers appeared. Both have been studied intensively not only on rings but also on other algebraic structures; for examples see \cite{bresarvukman,bresar,cusak,herstein,posner,vukman}, where more references can be found. 


It is a well-known fact that the indefinite integral of a function always gives rise to a family of functions and therefore it can be interpreted as the transformation, which maps each function $f$ to the family of functions $\{g \ | \ g'=f\}$. Therefore, to define integrations on rings as such set-valued transformations, the following notation is needed. For a ring $R$, let $2^R$ denote the family of all subsets of $R$. Then, an integration (a Jordan integration) $i$ on the ring $R$ will be defined as such (set-valued) transformation $i:R\to 2^R$. For any $A,B\in 2^R\setminus \{\emptyset \}$, we also define $A+B$ and $A\cdot B$ in a natural way as $A+B=\{a+b \ | \ a\in A, b\in B\}$ and $A\cdot B=\{a\cdot b \ | \ a\in A, b\in B\}$. In a special case when $A=\{a\}$, we write $a+B$ or $a\cdot B$ instead of $A+B$ or $A\cdot B$, respectively.

We proceed as follows. In the next section we introduce the notion of integration on a ring and show some of its basic properties. In Section \ref{4} we define a Jordan integration on a ring and show that every integration is also a Jordan integration. We also show that not every Jordan integration is an integration.

\section{Integrations on rings}\label{2}
In this section we introduce the new concept of integration on a ring and show many properties of such integrations. Our definition is motivated by the definition of an indefinite integral (a primitive function) in mathematical analysis, calculus.
\begin{definition}\label{def1}
Let $R$ be a ring, $x\in R$ and $d:R\to R$ a derivation on $R$. Let 
$$
i_d(x)=\{y\in R \ | \ x=d(y)\}.
$$
We say that $y\in R$ is a $d$-primitive element of the element $x\in R$, if 
$$
y\in i_d(x).
$$
We call the function $i_d:R\to 2^R$, $x\mapsto i_d(x)$, the $d$-integration on $R$ and the set $i_d(x)$ the $d$-integral of $x$. 

The function $i:R\to 2^R$ is an integration on $R$, if there is a derivation $d$ on $R$, such that $i=i_d$. For each $x\in R$ we call $i(x)$ an integral of $x$. 
\end{definition}
\begin{example}
Let $R$ be a ring and $d:R\to R$ the trivial derivation ($d(x)=0$ for all $x\in R$). Then $i_d(x)=R$ if $x=0$ and $i_d(x)=\emptyset$ if $x\neq 0$.
\end{example}

Many results that follow are easily obtained from the definition of an integration and the fact that any derivation $d:R\rightarrow R$ is a homomorphism
of the additive structure on $R$. Since the proofs of the results are short, we give them anyway.

The following results follow directly from Definition \ref{def1}.
\begin{proposition}\label{bla0}
Let $R$ be a ring and $d:R\to R$ a derivation on $R$. Then the following holds true.
\begin{enumerate}
\item $0\in i_d(0)$.
\item For each $x\in R$, $x\in i_d(d(x))$.
\item If $i_d(x)\neq \emptyset$, then $d(i_d(x))=\{x\}$ for all $x\in R$. 
\end{enumerate}
\end{proposition}
\begin{proof}
Since $d(0)=0$, it follows that $0\in i_d(0)$, which proves (1). 
To prove (2), let $y=d(x)$. Then it follows from Definition \ref{def1} that $x\in i_d(y)=i_d(d(x))$.
Finally we prove (3). Let $z\in i_d(x)$. Then $d(z)=x$ and hence $d(i_d(x))=\{x\}$.
\end{proof}
\begin{proposition}\label{bla6}
Let $R$ be a ring and $d:R\to R$ a derivation on $R$. Then 
\begin{enumerate}
\item $d$ is surjective if and only if $i_d(x)\neq \emptyset$ for all $x\in R$.
\item $d$ is injective if and only if $|i_d(x)|=1$ for all $x\in R$.
\end{enumerate}
\end{proposition}
\begin{proof}
First we prove (1). Suppose that $d$ is surjective and let $x\in R$. Then there is $y\in R$ such that $x=d(y)$. It follows that $y\in i_d(x)$ and $i_d(x)$ is not empty. 
Suppose that $i_d(x)\neq \emptyset$ for all $x\in R$. Let $x\in R$ and $y\in i_d(x)$. It follows that $x=d(y)$.

Next we prove (2). Suppose that $d$ is injective, $x\in R$, and $y,z\in i_d(x)$. Then $d(y)=d(z)=x$. Since $d$ is injective, it follows that $y=z$ and therefore $i_d(x)$ consists only of a single element. Suppose that $|i_d(x)|=1$ for all $x\in R$. Let $x,y,z\in R$ be such elements that $x=d(y)=d(z)$. Then $y,z\in i_d(x)$ and therefore $y=z$ (since $|i_d(x)|=1$).
\end{proof}
The theorems that follow describe basic properties of integrations, introduced in Definition \ref{def1}. First we show some properties of the integral $i_d(0)$ (Theorem \ref{bla1}). Note that $i_d(0)=\textup{Ker}(d)=\{x\in R \ | \ d(x)=0\}$ for any ring $R$ and any derivation $d$ on $R$.
\begin{definition}
Let $R$ be a ring with unity $1\in R$. For an integer $n$, $\mathbf{n}$ will always denote the element $\mathbf{n}=n\cdot 1=\underbrace{1+1+1+\ldots +1}_{n}\in R$. 
\end{definition}
\begin{theorem}\label{bla1}
Let $R$ be a ring with unity $1\in R$ and $d:R\to R$ a derivation on $R$. 
\begin{enumerate}
\item Then $-\mathbf{n},\mathbf{n}\in i_d(0)$ for each positive integer $n$.
\item If $\mathbf{n}$ is invertible, then 
$$
\mathbf{n}^{-1}\cdot \mathbf{m},\mathbf{m}\cdot \mathbf{n}^{-1}\in i_d(0)
$$ 
for all integers $m,n$.
\end{enumerate}
\end{theorem}
\begin{proof}
First we prove (1). From $d(1)=d(1\cdot 1)=d(1)\cdot 1+1\cdot d(1)=2d(1)$ it follows that $d(1)=0$. Therefore $1\in i_d(0)$. Since 
$$
d(\mathbf{n})=d(\underbrace{1+1+1+\ldots +1}_{n})=\underbrace{d(1)+d(1)+d(1)+\ldots +d(1)}_{n}=0,
$$ 
it follows that $\mathbf{n}\in i_d(0)$. Also, from $d(-\mathbf{n})=-d(\mathbf{n})=0$ it follows that $-\mathbf{n}\in i_d(0)$.

To prove (2), we first observe that $d(\mathbf{n}^{-1})=0$. This follows from the fact that $0=d(1)=d(\mathbf{n}\cdot \mathbf{n}^{-1})=d(\mathbf{n})\cdot \mathbf{n}^{-1}+\mathbf{n}\cdot d(\mathbf{n}^{-1})=0+\mathbf{n}\cdot d(\mathbf{n}^{-1})=\mathbf{n}\cdot d(\mathbf{n}^{-1})$, and since $\mathbf{n}$ is invertible, $d(\mathbf{n}^{-1})=0$. Next, it follows from $d(\mathbf{n}^{-1}\cdot \mathbf{m})=d(\mathbf{n}^{-1})\cdot \mathbf{m}+\mathbf{n}^{-1}\cdot d(\mathbf{m})=0+0=0$, that $\mathbf{n}^{-1}\cdot \mathbf{m}\in i_d(0)$. A similar proof gives $\mathbf{m}\cdot \mathbf{n}^{-1}\in i_d(0)$. 
\end{proof}
\begin{corollary}\label{bla10}
Let $R$ be a ring with unity $1\in R$ and $d:R\to R$ a derivation on $R$. If $y\in i_d(x)$ and $\mathbf{n}$ is invertible, then 
$$
y+\mathbf{n}^{-1}\cdot \mathbf{m},y+\mathbf{m}\cdot \mathbf{n}^{-1}\in i_d(x)
$$ 
for all integers $m,n$ and all $x,y\in R$.
\end{corollary}
\begin{proof}
The result follows from $d(y+\mathbf{n}^{-1}\cdot \mathbf{m})=d(y)+d(\mathbf{n}^{-1}\cdot \mathbf{m})=x+0=x$ by Theorem \ref{bla1}.
\end{proof}

 Next we generalize the result from calculus, saying that if $g$ is an indefinite integral of $f$, then also $g+c$ is an indefinite integral of $f$ for any constant function $c$. 

\begin{theorem}\label{jedro}
Let $R$ be a ring, $x,z\in R$, $d:R\to R$ a derivation on $R$ and  $y\in i_d(x)$. The following statements are equivalent.
\begin{enumerate}
\item $z\in i_d(x)$.
\item There is a unique element $w_0\in \textup{Ker}(d)$, such that $z=y+w_0$.
\end{enumerate}
\end{theorem}
\begin{proof}
Suppose first that $z\in i_d(x)$. Let $w_0=z-y$. Obviously, $z=y+w_0$ and $d(w_0)=d(z-y)=0$. If $z=y+w_0=y+w_1$ for $w_0,w_1\in \textup{Ker}(d)$, then $0=z-z=w_0-w_1$ and therefore $w_0=w_1$. 
For the converse, let $z=y+w_0$ for any $w_0\in \textup{Ker}(d)$. Then $d(z)=d(y+w_0)=d(y)+d(w_0)=x+0=x$ and therefore $z\in i_d(x)$.
\end{proof}
\begin{corollary}\label{jedro1}
Let $R$ be a ring, $x\in R$, $d:R\to R$ a derivation on $R$ and  $y\in i_d(x)$. Then 
$$
i_d(x)=\{y+z \ | \ z\in \textup{Ker}(d)\}.
$$
\end{corollary}
\begin{proof}
The result follows directly from Theorem \ref{jedro}. 
\end{proof}
\begin{example}\label{mc}
Let $D=\mathbb R\setminus \{\frac{(2k+1)\pi}{2} \ | \ k \textup{ is an integer}\}$ and let $R$ be the ring of all differentiable functions on $D$ with the standard addition and multiplication of functions, i.e.\ $(f+g)(x)=f(x)+g(x)$ and $(f\cdot g)(x)=f(x)g(x)$ for each $x\in D$. Let $d$ be the standard derivation on $D$, i.e.\ $d(f)=f'$ for each $f\in R$. Next, let $h\in R$ be defined by 
$$
h(x)=\frac{(2k+1)\pi}{2}, \textup{ if } x\in \left(\frac{(2k-1)\pi}{2} ,\frac{(2k+1)\pi}{2}\right).
$$ 
Then clearly $h\in \textup{Ker}(d)$. Finally, let $f,g\in R$ be defined by $f(x)=\frac{1}{\cos^2x}$ and $g(x)=\tan x$ for all $x\in D$. Since $g\in i_d(f)$, it follows from Theorem \ref{jedro}, that also $g+h\in i_d(f)$.  
\end{example}
Next we realize the well-known formula 
$$
\int c\cdot f(x)~\textup{d}x=c\cdot \int f(x)~\textup{d}x
$$
from calculus for arbitrary rings.
\begin{theorem}\label{mcmc}
Let $R$ be a ring, $d:R\to R$ a derivation on $R$ and $w\in \textup{Ker}(d)$. Then the following holds true.
\begin{enumerate}
\item $w\cdot i_d(x)\subseteq i_d(w\cdot x)$ and $i_d(x)\cdot w\subseteq i_d(x\cdot w)$.
\item If $i_d(x)\neq \emptyset$, then $i_d(w\cdot x),i_d(x\cdot w)\neq \emptyset$.
\item If $R$ is a ring with unity $1\in R$, $w$ an invertible element and $i_d(x)\neq \emptyset$, then $i_d(w\cdot x)=w\cdot i_d(x)$ and $i_d(x\cdot w)=i_d(x)\cdot w$.
\end{enumerate}  
\end{theorem}
\begin{proof}
If $i_d(x)=\emptyset$, then obviously (1) holds true. Suppose that $i_d(x)\neq\emptyset$. We only prove that $w\cdot i_d(x)\subseteq i_d(w\cdot x)$ ($i_d(x)\cdot w\subseteq i_d(x\cdot w)$ can be proved similarly). Let $z\in i_d(x)$. Then $d(w\cdot z)=d(w)\cdot z+w\cdot d(z)=0+w\cdot x=w\cdot x$ and therefore $w\cdot z\in i_d(w\cdot x)$. This proves (1) and (2). To prove (3), suppose that $w$ is invertible and $i_d(x)\neq \emptyset$. First observe, that from $0=d(1)=d(w\cdot w^{-1})=d(w)\cdot w^{-1}+w\cdot d(w^{-1})=w\cdot d(w^{-1})$ we get $d(w^{-1})=0$. It also follows from (2) that $i_d(w\cdot x)\neq \emptyset$. Take any $y\in i_d(w\cdot x)$ and let $z=w^{-1}\cdot y$. Then 
$$
d(z)=d(w^{-1}\cdot y)=d(w^{-1})\cdot y+w^{-1}\cdot d(y)=0+w^{-1}\cdot w\cdot x=x.
$$ 
This means that $z\in i_d(x)$ and therefore $y=w\cdot z\in w\cdot i_d(x)$. A similar proof gives $i_d(x\cdot w)=i_d(x)\cdot w$.
\end{proof}
\begin{example}
Applying Theorem \ref{mcmc} to Example \ref{mc}, one gets $i_d(h\cdot f)=h\cdot i_d(f)$. Presenting this with the standard integral notation, one gets:
$$
\int h(x)f(x)~\textup{d}x=h(x)\int f(x)~\textup{d}x.
$$ 
Note that $h$ is not a constant function (but it is constant on every connected component of the domain $D$).
\end{example}
\begin{remark}
Let $R$ be a ring with unity $1\in R$. If $w$ is not an invertible element, then $i_d(x)\neq \emptyset$ does not necessarily imply $i_d(w\cdot x)=w\cdot i_d(x)$ or $i_d(x\cdot w)=i_d(x)\cdot w$. 
This follows immediately from the fact that $\int w(x) ~\textup{d}x\neq w(x)\cdot \int ~\textup{d}x$, for $w(x)=0$ for each $x\in \mathbb R$.
\end{remark}
Next, for $y\in i(x)$, $y_1\in i(x_1)$, $y_2\in i(x_2)$, we describe in Theorem \ref{bla7} the integrals that contain $y_1+y_2$, $y_1\cdot y_2$, and $y^{-1}$, respectively.
\begin{theorem}\label{bla7}
Let $R$ be a ring and $d:R\to R$ a derivation on $R$. 
\begin{enumerate}
\item If $y_1\in i_d(x_1)$ and $y_2\in i_d(x_2)$, then $y_1+y_2\in i_d(x_1+x_2)$ for all $x_1,x_2,y_1,y_2\in R$.
\item If $y_1\in i_d(x_1)$ and $y_2\in i_d(x_2)$, then $y_1\cdot y_2\in i_d(x_1\cdot y_2+y_1\cdot x_2)$ for all $x_1,x_2,y_1,y_2\in R$.
\item If $R$ is a ring with unity $1\in R$ and $y\in i_d(x)$, then $y^{-1}\in i_d(-y^{-1}\cdot x\cdot y^{-1})$ for all $x\in R$ and all invertible $y\in R$. 
\end{enumerate}
\end{theorem}
\begin{proof}
(1) follows from $d(y_1+y_2)=d(y_1)+d(y_2)=x_1+x_2$, (2) follows from $d(y_1\cdot y_2)=d(y_1)\cdot y_2+y_1\cdot d(y_2)=x_1\cdot y_2+y_1\cdot x_2$. Finally we prove (3). It follows from $0=d(1)=d(y\cdot y^{-1})=d(y)\cdot y^{-1}+y\cdot d(y^{-1})$ that $d(y^{-1})=-y^{-1}\cdot x\cdot y^{-1}$. Therefore $y^{-1}\in i_d(-y^{-1}\cdot x\cdot y^{-1})$. 
\end{proof}
\begin{corollary}\label{bla7c}
Let $R$ be a ring and $d:R\to R$ a derivation on $R$. If $y,z\in i_d(x)$, then 
\begin{enumerate}
\item $y+z\in i_d(2x)$ and
\item $y\cdot z\in i_d(x\cdot z+y\cdot x)$
\end{enumerate}
 for all $x,y,z\in R$. 
\end{corollary}
\begin{proof}
The results follow directly from Theorem \ref{bla7}. 
\end{proof}
\begin{corollary}\label{bla13c}
If $R$ be a commutative ring with unity $1\in R$ and $d:R\to R$ a derivation on $R$. If $y\in i_d(x)$ then $y^{-1}\in i_d(-y^{-2}\cdot x)$ for all $x\in R$ and all invertible $y\in R$. 
\end{corollary}
\begin{proof}
By Theorem \ref{bla7}, $y^{-1}\in i_d(-y^{-1}\cdot x\cdot y^{-1})$. Since $R$ is commutative, it follows that $-y^{-1}\cdot x\cdot y^{-1}=-y^{-2}\cdot x$. Therefore $y^{-1}\in i_d(-y^{-2}\cdot x)$. 
\end{proof}
We use the following lemma to prove Theorem \ref{bla4}, where the integration by parts is introduced -- it generalizes the integration by parts of functions $\int u(x)v'(x)~\textup{d}x=u(x)v(x)-\int u'(x)v(x)~\textup{d}x$, or more compactly
$$
uv=\int u~\textup{d}v+\int v~\textup{d}u.
$$
\begin{lemma}\label{bla3}
Let $R$ be a ring and $d:R\to R$ a derivation on $R$. If $i_d(x),i_d(y)\neq \emptyset$, then $i_d(x+y)=i_d(x)+i_d(y)$ for all $x,y\in R$.
\end{lemma}
\begin{proof}
Let $z\in i_d(x)+i_d(y)$. Then there are $a,b\in R$, $a\in i_d(x)$, $b\in i_d(y)$, such that $z=a+b$. Then $d(z)=d(a+b)=d(a)+d(b)=x+y$ and therefore $z\in i_d(x+y)$ (this also proves that from $i_d(x),i_d(y)\neq \emptyset$ it follows that $i_d(x+y)\neq \emptyset$). 

Let $z\in i_d(x+y)$. Then $d(z)=x+y$. Since $i_d(x)\neq \emptyset$, let $a\in i_d(x)$. Then $d(z-a)=d(z)-d(a)=d(z)-x=y$ and therefore $z-a\in i_d(y)$. We have proved that $z=a+(z-a)\in i_d(x)+i_d(y)$.
\end{proof}
\begin{theorem}\label{bla4}(Integration by parts)
Let $R$ be a ring and $d:R\to R$ a derivation on $R$. If $i_d(d(x)\cdot y),i_d(x\cdot d(y))\neq \emptyset$, then $x\cdot y\in i_d(d(x)\cdot y)+i_d(x\cdot d(y))$ for all $x,y\in R$. 
\end{theorem}
\begin{proof}
It follows from Proposition \ref{bla0} and the definition of a derivation that $x\cdot y\in i_d(d(x\cdot y))=i_d(d(x)\cdot y+x\cdot d(y))$ and therefore $x\cdot y\in i_d(d(x)\cdot y+x\cdot d(y))$. By Lemma \ref{bla3}, $i_d(d(x)\cdot y+x\cdot d(y))=i_d(d(x)\cdot y)+i_d(x\cdot d(y))$.
\end{proof}
As seen in Theorem \ref{bla4}, the integration by parts only works through if 
$$
i_d(d(x)\cdot y),i_d(x\cdot d(y))\neq \emptyset.
$$ 
On the other hand, the set $i_d(d(x)\cdot y+x\cdot d(y))$ is never empty (we can always find $x\cdot y$ in it). This gives rise to the following question. Is it possible that $i_d(d(x)\cdot y),i_d(x\cdot d(y))= \emptyset$? The following example demonstrates that the answer to the question is affirmative.
\begin{example}\label{matrike}
Let $R$ be the ring of real $2\times 2$-matrices $M_2(\mathbb R)$, $A=
\begin{bmatrix}
 1&0\\
 0&0
\end{bmatrix}
$
and let $d:R\to R$ be the inner derivation on $R$, defined by $d(X)=X\cdot A-A\cdot X$. Next, let 
$
X=\begin{bmatrix}
 0&1\\
 0&0
\end{bmatrix}
$
and 
$
Y=\begin{bmatrix}
 0&0\\
 1&0
\end{bmatrix}
$.
Then 
$
d(X)=-X
$, 
$
d(Y)=Y
$.
It follows from 
$$
d\left(\begin{bmatrix}
 a&b\\
 c&d
\end{bmatrix}\right)=\begin{bmatrix}
 0&-b\\
 c&0
\end{bmatrix},
$$ 
 $d(X)\cdot Y=-A$ and $X\cdot d(Y)=A$, that $i_d(d(X)\cdot Y)= \emptyset$ and $i_d(X\cdot d(Y))= \emptyset$.
\end{example}
We conclude the section with theorems that present results generalizing the formula 
$$
\int x^n~\textup{d}x=\frac{x^{n+1}}{n+1}+c
$$ 
from calculus. First we state and prove the following lemma.
\begin{lemma}\label{prenesi1}
Let $R$ be a ring with unity $1\in R$ and $d:R\to R$ a derivation on $R$. 
\begin{enumerate}
\item If $x\in i_d(\mathbf{n}\cdot y)$ and $\mathbf{n}$ is invertible, then $\mathbf{n}^{-1}\cdot x\in i_d(y)$ for all $x,y\in R$ and all integers $n$.
\item If $\mathbf{n}\cdot x\in i_d(y)$ and $\mathbf{n}$ is invertible, then $x\in i_d(\mathbf{n}^{-1}\cdot y)$ for all $x,y\in R$ and all integers $n$.
\end{enumerate}
\end{lemma}
\begin{proof}
It follows from $d(\mathbf{n}^{-1}\cdot x)=d(\mathbf{n}^{-1})\cdot x+\mathbf{n}^{-1}\cdot d(x)=0+\mathbf{n}^{-1}\cdot d(x)=\mathbf{n}^{-1}\cdot d(x)=\mathbf{n}^{-1}\cdot \mathbf{n}\cdot y=y$ that $\mathbf{n}^{-1}\cdot x\in i_d(y)$. This proves (1). Next, from $d(x)=d(\mathbf{n}^{-1}\cdot \mathbf{n}\cdot x)=d(\mathbf{n}^{-1})\cdot (\mathbf{n}\cdot x)+\mathbf{n}^{-1}\cdot d(\mathbf{n}\cdot x)=0+\mathbf{n}^{-1}\cdot d(\mathbf{n}\cdot x)=\mathbf{n}^{-1}\cdot y$ it follows that $x\in i_d(\mathbf{n}^{-1}\cdot y)$. This proves (2) and we are done.
\end{proof}
\begin{theorem}\label{bla11}
Let $R$ be a commutative ring with unity $1\in R$ and $d:R\to R$ a derivation on $R$. Then 
$$
x^n\in i_d(\mathbf{n}\cdot x^{n-1}\cdot d(x))
$$
for all $x\in R$ and all positive integers $n$.
\end{theorem}
\begin{proof}
We prove the theorem by induction on $n$.
\begin{itemize}
\item $n=1$. Obviously $x\in i_d(d(x))$ by Proposition \ref{bla0}.

\item $n=2$. We prove that $x^2\in i_d(x\cdot d(x))$. Since $d(x^2)=d(x)\cdot x+x\cdot d(x)=\mathbf{2}\cdot x\cdot d(x)$, it follows that $x^2\in i_d(\mathbf{2}\cdot x\cdot d(x))$.

\item Suppose that $x^n\in i_d(\mathbf{n}\cdot x^{n-1}\cdot d(x))$. We prove that $x^{n+1}\in i_d(\mathbf{n}+1)\cdot x^{n}\cdot d(x))$. Observe that $d(x^{n+1})=d(x\cdot x^n)=d(x)\cdot x^{n}+x\cdot d(x^{n})=d(x)\cdot x^{n}+x\cdot (\mathbf{n}\cdot x^{n-1}\cdot d(x))=x^n\cdot d(x)+\mathbf{n}\cdot x^n\cdot d(x)=(\mathbf{n}+1)\cdot x^n\cdot d(x)$. Therefore $x^{n+1}\in i_d((\mathbf{n}+1)\cdot x^{n}\cdot d(x))$.
\end{itemize}
\end{proof}
\begin{corollary}\label{bla11c}
Let $R$ be a commutative ring with unity $1\in R$ and $d:R\to R$ a derivation on $R$. If $\mathbf{n}$ is invertible, then
$$
\mathbf{n}^{-1}\cdot x^{n}\in i_d(x^{n-1}\cdot d(x))
$$
for all $x\in R$ and all positive integers $n$. 
\end{corollary}
\begin{proof}
The result follows directly from Theorem \ref{bla11} and Lemma \ref{prenesi1}.
\end{proof}
\begin{theorem}\label{bla12}
Let $R$ be a commutative ring with unity $1\in R$ and $d:R\to R$ a derivation on $R$. Then
$$
x^{-n}\in i_d(-\mathbf{n}\cdot x^{-n-1}\cdot d(x))
$$
for all invertible $x\in R$ and all positive integers $n$. 
\end{theorem}
\begin{proof}
We prove the theorem by induction on $n$.
\begin{itemize}
\item $n=1$. It follows from $0=d(1)=d(x\cdot x^{-1})=d(x)\cdot x^{-1}+x\cdot d(x^{-1})$ that $d(x^{-1})=-x^{-2}\cdot d(x)$. Therefore $x^{-1}\in i_d(-x^{-2}\cdot d(x))$.

\item Suppose that $x^{-n}\in i_d(-\mathbf{n}\cdot x^{-n-1}\cdot d(x))$. We prove that $x^{-n-1}\in i_d((-\mathbf{n}-1)\cdot x^{-n-2}\cdot d(x))$. Observe that $d(x^{-n-1})=d(x^{-n}\cdot x^{-1})=d(x^{-n})\cdot x^{-1}+x^{-n}\cdot d(x^{-1})=(-\mathbf{n}\cdot x^{-n-1}\cdot d(x))\cdot x^{-1}+x^{-n}\cdot (-x^{-2}\cdot d(x))=-\mathbf{n}\cdot x^{-n-2}\cdot d(x)-x^{-n-2}\cdot d(x)=(-\mathbf{n}-1)\cdot x^{-n-2}\cdot d(x)$. Therefore $x^{-n-1}\in i_d((-\mathbf{n}-1)\cdot x^{-n-2}\cdot d(x))$.
\end{itemize}
\end{proof}
\begin{corollary}\label{bla12c}
Let $R$ be a commutative ring with unity $1\in R$ and $d:R\to R$ a derivation on $R$. Then 
$$
x^{n}\in i_d(\mathbf{n}\cdot x^{n-1}\cdot d(x))
$$
for all invertible $x\in R$ and all integers $n$. 
\end{corollary}
\begin{proof}
If $n\neq 0$, the result follows from Theorem \ref{bla11} and Theorem \ref{bla12}. If $n=0$, then $x^0=1\in i_d(0)$ by Theorem \ref{bla1}.
\end{proof}
\begin{corollary}\label{bla12cc}
Let $R$ be a commutative ring with unity $1\in R$ and $d:R\to R$ a derivation on $R$. If $\mathbf{n}$ is invertible, then
$$
\mathbf{n}^{-1}\cdot x^{n}\in i_d(x^{n-1}\cdot d(x))
$$
for all invertible $x\in R$ and all integers $n$. 
\end{corollary}
\begin{proof}
The result follows directly from Corollary \ref{bla12c} and Lemma \ref{prenesi1}.
\end{proof}


Let $\mathcal R_d=\{i_d(x) \ | \ i_d(x)\neq \emptyset, x\in R\}$. Since $\mathcal R_d=R/_{\textup{Ker} d}$, it follows from the first isomorphism theorem that $\mathcal R_d$ is, as an additive group, isomorphic to $d(R)$. Therefore the ring structure on  $\mathcal R_d$ can be easily obtained from $d(R)$. However, $\mathcal R_d$ can be interpreted as a ring if $d(R)$ is a subring of $R$. As seen in Example \ref{not}, $d(R)$ is not necessarily a subring of $R$.

\begin{definition}\label{proper}
Let $d$ be a derivation on a ring $R$. The derivation $d$ is a proper derivation, if $d(R)$ is a subring of $R$.
\end{definition}

We conclude the section with several interesting examples of proper  as well as non-proper derivations.
\begin{example}
Let $R$ be a ring and $d:R\to R$ the trivial derivation on $R$. Obviously $d$ is a proper derivation. 
\end{example}
\begin{example}\label{not}
Let $R$, $d$, and $A,X,Y\in R$ be such as in Example \ref{matrike}. 
Then $X\cdot Y=A$, $d(-X)=X$ and $d(Y)=Y$,
and therefore (by a similar argument as the one in Example \ref{matrike}) one can easily see that $i_d(X\cdot Y)=\emptyset$ although $i_d(X),i_d(Y)\neq \emptyset$. Hence, $d$ is not a proper derivation.
\end{example} 
\begin{example}\label{polinom}
Let $R$ be the polynomial ring $\mathbb R[X]$ and $d:R\to R$ the standard derivation $d(p)=p'$. Then obviously, for any  $p_1,p_2\in R$, $p_1\cdot p_2\in R$ and there always exists a polynomial $p\in R$ such that $d(p)=p_1\cdot p_2$. Therefore $d$ is a proper derivation. 
\end{example}
For derivations $d$ in Examples \ref{polinom}, $d(R)$ a subring of $R$ but it is not a proper subring of $R$. The following example presents a ring $R$ and a non-trivial inner derivation $d$ on $R$ such that $d(R)$ is a proper subring of $R$.
\begin{example}\label{tri}
Let 
$$
R=\left\{\begin{bmatrix}
a&b&c\\
0&0&d\\
0&0&e
\end{bmatrix} \ | \ a,b,c,d,e\in \mathbb R\right\}
$$
be the subring of real $3\times 3$-matrices $M_3(\mathbb R)$, 
and let 
$A=
\begin{bmatrix}
1&1&1\\
0&0&1\\
0&0&1
\end{bmatrix}$.
Then let $d:R\to R$ be the inner derivation on $R$, defined by $d(X)=X\cdot A-A\cdot X$. It is easily seen that 
$$
d(R)=\left\{\begin{bmatrix}
0&a&b\\
0&0&c\\
0&0&0
\end{bmatrix} \ | \ a,b,c\in \mathbb R\right\}
$$
is a proper subring of $R$. Therefore $d$ is a proper derivation on $R$.
\end{example}

\section{Jordan integrations on rings}\label{4}

In this section we introduce the new concept of Jordan integrations on rings and show basic properties of such Jordan integrations.
\begin{definition}\label{defdef1}
Let $R$ be a ring, $x\in R$ and $\delta:R\to R$ a Jordan derivation on $R$. Let 
$$
j_{\delta}(x)=\{y\in R \ | \ x=\delta(y)\}.
$$
We say that $y\in R$ is a Jordan $\delta$-primitive element of the element $x\in R$, if  
$$
y\in j_{\delta}(x).
$$
We call the function $j_{\delta}:R\to 2^R$, $x\mapsto j_{\delta}(x)$, the Jordan ${\delta}$-integration on $R$ and the set $j_{\delta}(x)$ the Jordan ${\delta}$-integral of $x$. 

The function $j:R\to 2^R$ is a Jordan integration on $R$, if there is a Jordan derivation ${\delta}$ on $R$, such that $j=j_{\delta}$. For each $x\in R$ we call $j(x)$ a Jordan integral of $x$. 
\end{definition}
Since every derivation $d$ is a Jordan derivation, it follows that any $d$-integration is a Jordan $d$-integration. Since there are rings $R$ and Jordan derivations ${\delta}$ on $R$ that are not derivations, it follows that for such ${\delta}$'s, the corresponding Jordan ${\delta}$-integrations $j_{\delta}$ are never ${\delta}$-integrations. But, is it true that for any such $\delta$ there is a derivation $d$, such that $j_{\delta}=i_d$? The following theorem answers the question in negative.
\begin{theorem}
There is a ring $R$ and a Jordan derivation $\delta$ on $R$ such that for each derivation $d$ on $R$, $i_d\neq j_{\delta}$. 
\end{theorem}
\begin{proof}
Let $R$ be a ring and $\delta$ a Jordan derivation on $R$ that is not a derivation. Assume that there is a derivation $d$ on $R$ such that $i_d(x)=j_{\delta}(x)$ for all $x\in R$. Let $y\in R$ be any element. Then there is $x\in R$ such that $y\in i_d(x)=j_{\delta}(x)$. It follows that $d(y)=x=\delta(y)$. Therefore $d=\delta$, which is a contradiction. 
\end{proof}

Following the same line of thoughts as in Section \ref{2}, one can obtain similar results for Jordan integrations -- they are listed in the theorem below. 
\begin{theorem}\label{bla}
Let $R$ be a ring and $\delta:R\to R$ a Jordan derivation on $R$. Then the following holds true.
\begin{enumerate}
\item $0\in j_{\delta}(0)$.
\item If $y,z\in j_{\delta}(x)$, then $y-z\in j_{\delta}(0)$ for all $x,y,z\in R$.
\item For all $x\in R$, $x\in j_{\delta}({\delta}(x))$.
\item If $x\in R$ and  $y\in j_{\delta}(x)$, then 
$$
j_{\delta}(x)=\{y+z \ | \ z\in \textup{Ker}(\delta)\}.
$$
\item If $j_{\delta}(x)\neq \emptyset$, then $\delta(j_{\delta}(x))=\{x\}$ for all $x\in R$.
\item If $j_{\delta}(x),j_{\delta}(y)\neq \emptyset$, then $j_{\delta}(x+y)=j_{\delta}(x)+j_{\delta}(y)$ for all $x,y\in R$.
\item (Jordan integration by parts) 

\noindent If 
$
j_{\delta}(\delta(x)\cdot y),j_{\delta}(x\cdot \delta(y)),j_{\delta}(\delta(y)\cdot x),j_{\delta}(y\cdot \delta(x))\neq \emptyset
$, then  
$$
x\circ y\in j_{\delta}(\delta(x)\cdot y)+j_{\delta}(x\cdot \delta(y))+j_{\delta}(\delta(y)\cdot x)+j_{\delta}(y\cdot \delta(x))
$$
for all $x,y\in R$.
\item If $y_1\in j_{\delta}(x_1)$ and $y_2\in j_{\delta}(x_2)$, then $y_1+y_2\in j_{\delta}(x_1+x_2)$ for all $x_1,x_2,y_1,y_2\in R$.
\item If $y_1\in j_{\delta}(x_1)$ and $y_2\in j_{\delta}(x_2)$, then $y_1\circ y_2\in j_{\delta}(x_1\cdot y_2+y_1\cdot x_2+x_2\cdot y_1+y_2\cdot x_1)$ for all $x_1,x_2,y_1,y_2\in R$.
\item The Jordan derivation $\delta$ is surjective if and only if $j_{\delta}(x)\neq \emptyset$ for all $x\in R$.
\item The Jordan derivation $\delta$ is injective if and only if $|j_{\delta}(x)|=1$ for all $x\in R$. 
\end{enumerate}
\end{theorem}
\begin{proof}
Following the similar ideas as those in the proofs of corresponding  theorems in Section \ref{2}, one can easily prove the theorem. We leave the details to a reader.
\end{proof}

\bibliographystyle{amsplain}

\begin{thebibliography}{9}
\bibitem{benkovic}  D.~Benkovi\v c and D.~Eremita, Multiplicative Lie n-derivations of triangular rings, {\em Linear Algebra Appl.\ } {\bf 436} (2012), 4223--4240.   
\bibitem{benkovic1} D.~Benkovi\v c and N.~\v Sirovnik, Jordan derivations of unital algebras with idempotents, {\em  Linear Algebra Appl.} {\bf 437} (2012), 2271--2284.  
\bibitem{bresarvukman} M.~Bre\v sar and J.~Vukman, Jordan derivations on prime rings, {\em Bull. Austral. Math. Soc.\ } {\bf 37} (1988), 321--322.
\bibitem{bresarvukman1} M.~Bre\v sar and J.~Vukman, On left derivations and related mappings, {\em Proc. Amer. Math. Soc.\ } {\bf 110} (1990), 7--16.
\bibitem{bresarvukman2} M.~Bre\v sar and J.~Vukman, Jordan $(\theta ,\phi )-$derivations, {\em Glasnik Mat.\ } {\bf 16} (1991), 13--17.
\bibitem{bresar} M.~Bre\v sar, Jordan derivations on semiprime rings, {\em Proc. Amer. Math. Soc.\ } {\bf 104} (1988), 1003--1006.
\bibitem{bresar0} M.~Bre\v sar, On the distance of the compositum of two derivations to the generalized derivations, {\em Glasgow Math. J.\ } {\bf 33} (1991), 89--93.
\bibitem{bresar1} M.~Bre\v sar, Commuting traces of biadditive mappings, commutativity-preserving mappings and Lie mappings, {\em Trans. Amer. Math. Soc.\ } {\bf 335} (1993), 525--546.
\bibitem{cusak} J.~Cusak, Jordan derivations on rings, {\em Proc. Amer. Math. Soc.\ } {\bf 53} (1975), 321--324.
\bibitem{fosner} M.~Fo\v{s}ner and J.~Vukman, Identities with generalized derivations in prime rings, {\em Mediteranean J. Math.\ } {\bf 9} (2012), 847--863.
\bibitem{herstein} N.~Herstein, Jordan derivations of prime rings, {\em Proc. Amer. Math. Soc.\ } {\bf 8} (1957), 1104--1110.
\bibitem{hvala} B.~Hvala, Generalized derivations in rings, {\em Comm. Algebra} {\bf 26} (1998), 1147--1166.
\bibitem{hvala1} B.~Hvala, Generalized Lie derivations in prime rings, {\em Taiwanese J. Math.\ } {\bf 11} (2007), 1425--1430.
\bibitem{liu} C.-K.~Liu and W.-K.~Shiue, Generalized Jordan triple $(\theta ,\phi )-$derivations on semiprime rings, {\em Taiwanese J. Math.\ } {\bf 11} (2007), 1397--1406.
\bibitem{posner} E.~C.~Posner, Derivations in prime rings, {\em Proc. Amer. Math. Soc.\ } {\bf 8} (1957), 1093--1100.
\bibitem{vukman} J.~Vukman, Identities with derivations on rings and Banach algebras, {\em Glas. Mat. Ser. III } {\bf 40} (2005), 189--199.
\bibitem{vukman2} J.~Vukman, On left Jordan derivations of rings and Banach lgebras, {\em Aequat. Math.\ } {\bf 75} (2008), 260--266.
\end{thebibliography}

\noindent Iztok Bani\v c \\ Faculty of Natural Sciences and Mathematics \\University of Maribor \\ Koro\v ska 160, SI-2000 Maribor \\ Slovenia \\ and \\ Andrej Maru\v si\v c Insitute \\ University of Primorska\\ Muzejski trg 2, SI-6000 Koper\\ Slovenia
\\ E-mail address: iztok.banic@um.si

\end{document}